\documentclass[12pt]{article}
\usepackage[latin1]{inputenc}
\usepackage[english]{babel}
\usepackage{graphicx}
\usepackage{amsmath,amsfonts,amssymb,amsthm,amsopn,amscd}
\usepackage{bbm,wasysym,stmaryrd}
\usepackage{subfigure}
\usepackage{color}
\usepackage{setspace}
\usepackage{authblk}
\usepackage{hyperref}
\usepackage[normalem]{ulem}

\newcommand{\R}{\mathbbm{R}}

\newcommand{\B}{\mathcal{B}}

\newcommand{\F}{\mathcal{F}}
\newcommand{\G}{\mathcal{G}}
\newcommand{\Hh}{\mathcal{H}}

\newcommand{\Exp}{\mathbb{E}} 

\newcommand{\T}{\mathcal{T}}  

\newcommand{\n}{\mathfrak{n}}

\newcommand{\D}{\mathcal{D}}
\newcommand{\M}{\mathcal{M}}
\newcommand{\MM}{\mathfrak{M}}
\newcommand{\C}{\mathcal{C}}

\newcommand{\norm}[1]{\| #1 \|}
\newcommand{\lsup}[1]{\underset{#1\to\infty}{\overline{\lim}}}
\newcommand{\linf}[1]{\underset{#1\to\infty}{\underline{\lim}}}

\newcommand{\J}{\mathbf J}
\newcommand{\undt}[1]{\underline{#1}}

\theoremstyle{plain}
\newtheorem{theorem}{Theorem}
\newtheorem{corollary}[theorem]{Corollary}
\newtheorem{proposition}[theorem]{Proposition}
\newtheorem{lemma}[theorem]{Lemma}

\theoremstyle{definition}
\newtheorem{definition}{Definition}

\newenvironment{remark}{\par {\noindent \it \sc Remark.} \small \it } {}

\begin{document}
\title{On the form of the relative entropy between measures on the space of continuous functions}

\author[*]{James MacLaurin}
\author[*]{Olivier Faugeras}
\affil[*]{NeuroMathComp Laboratory, INRIA Sophia-Antipolis, 2004 route des Lucioles, B.P. 93, 06902 Sophia-Antipolis, France
  \hspace{2cm} {\tt \small james.maclaurin@inria.fr}}

\maketitle

\abstract{
In this paper we derive an integral (with respect to time) representation of the relative entropy (or Kullback-Leibler Divergence) $R(\mu || P)$, where $\mu$ and $P$ are measures on $C([0,T];\R^d)$. The underlying measure $P$ is a weak solution to a Martingale Problem with continuous coefficients. Since $R(\mu|| P)$ governs the exponential rate of convergence of the empirical measure (according to Sanov's Theorem), this representation is of use in the numerical and analytical investigation of finite-size effects in systems of interacting diffusions.}

\section{Introduction}

In this paper we derive an integral representation of the relative entropy $R(\mu || P)$, where $\mu$ is a measure on $C([0,T];\R^d)$ and $P$ governs the solution to a stochastic differential equation (SDE).  The relative entropy is used to quantify the distance between two measures. It has considerable applications in statistics, information theory and communications. It has been used in the long-time analysis of Fokker-Planck equations \cite{plastino-miller-etal:97,desvillettes-villani:01}, the analysis of dynamical systems \cite{yu-mehta:09} and the analysis of spectral density functions \cite{georgiou-lindquist:03}. It has been used in financial mathematics to quantify the difference between martingale measures \cite{fritelli:00,grandits-rheinlander:02}. The finiteness of $R(\mu|| P)$ has been shown to be equivalent to the invertibility of certain shifts on Wiener Space, when $P$ is the Wiener Measure \cite{ustunel:09,lassalle:12}.

One of the most important applications of the relative entropy is in the area of Large Deviations. Sanov's Theorem dictates that the empirical measure induced by independent samples governed by the same probability law $P$ converge towards their limit exponentially fast; and the constant governing the rate of convergence is the relative entropy \cite{dembo-zeitouni:97}. Large Deviations have been applied for example to spin glasses \cite{ben-arous-guionnet:95}, neural networks \cite{moynot-samuelides:02,faugeras-maclaurin:13b} and mean-field models of risk \cite{josselin-garnier-yang:13}. In the mean-field theory of neuroscience in particular, there has been a recent interest in the modelling of `finite-size-effects' \cite{baladron-fasoli-etal:11}. Large Deviations provides a mathematically rigorous tool to do this. However when $P$ governs a diffusion process the calculation of $R(\mu || P)$ for an arbitrary measure $\mu$ is not necessarily straightforward. In this context, the variational definition in Definition \ref{eq:relativeentropy} below is not straightforward to work with, and neither is the equivalent expression $E^{\mu}[\log\frac{d\mu}{dP}]$. The problem becomes particularly acute when one wishes to numerically calculate the rate function.

The classic paper by \cite{dawson-gartner:87} obtained a Large Deviations estimate for the empirical process corresponding to a system of diffusion processes, with coefficients themselves dependent on the empirical process. The empirical process in this paper is a map from the time $t$ to the empirical measure at time $t$. They used projective limits to determine a representation of the rate function as an integral (analogously to Freidlin and Wentzell \cite{freidlin-wentzell:98}) with respect to time.  This rate function is not the relative entropy because the authors' empirical process differs from the empirical measure. Various authors have extended this work. Most recently, \cite{budhiraja-dupuis-etal:12,fischer:12} have used weak convergence and stochastic optimal control techniques to obtain a Large Deviation Principle for the empirical measure corresponding to a similar system to \cite{dawson-gartner:87} (although in a more general form). This empirical measure is a measure on the space $C([0,T](\R^d))$ of continuous functions; it contains more information about the system than the empirical process of \cite{dawson-gartner:87}. The rate function is in variational form, being equal to the minimum of a cost functional over a set of control processes. If one assumes that the coefficients governing the diffusion process are independent of the empirical measure in this paper, then through Sanov's Theorem one may infer that the rate function is equal to the relative entropy: in other words, one may infer a variational representation of the relative entropy from these papers. It is also of interest to consult \cite{boue-dupuis:98}, who were the first to prove such a variational representation of the relative entropy in the case where the underlying process is a Wiener process.

In this paper we derive a specific integral (with respect to time) representation of the relative entropy $R(\mu|| P)$ when $P$ is a diffusion process. This $P$ is the same as in \cite[Section 4]{dawson-gartner:87}. The representation makes use of regular conditional probabilities. It ought, in many circumstances, to be more numerically tractable than the standard definition in \ref{eq:relativeentropy}, and thus it would be of practical use in the applications listed above.
\section{Outline of Main Result}

Let $\T$ be the Banach Space $C([0,T];\R^d)$ equipped with the norm
\begin{equation}
\norm{X} = \sup_{s\in [0,T]}\lbrace |X_s| \rbrace,
\end{equation}
where $\left| \cdot \right|$ is the standard Euclidean norm over $\R^d$. We let $(\F_t)$ be the canonical filtration over $(\T,\B(\T))$. For some topological space $\mathcal{X}$, we let $\mathcal{B}(\mathcal{X})$ be the Borelian $\sigma$-algebra and $\M(\mathcal{X})$ the space of all probability measures on $(\mathcal{X},\mathcal{B}(\mathcal{X}))$. Unless otherwise indicated, we endow $\M(\mathcal{X})$ with the topology of weak convergence. Let $\sigma = \lbrace t_1,t_2,\ldots,t_m \rbrace$ be a finite set of elements such that $t_1\geq 0$, $t_m \leq T$ and $t_{j} < t_{j+1}$.  We term $\sigma$ a partition, and denote the set of all such partitions by  $\J$. The set of all partitions of the above form such that $t_1 = 0$ and $t_m = T$ is denoted $\J_*$. For some $t\in [0,T]$, we define $\underline{\sigma}(t) = \sup\lbrace s\in \sigma | s\leq t\rbrace$. 

\begin{definition}\label{eq:relativeentropy}
Let $(\Omega,\Hh)$ be a measurable space, and $\mu$, $\nu$ probability measures. 
\begin{equation*}
R_{\Hh}(\mu||\nu) = \sup_{f\in \mathcal{E}}\left\lbrace \Exp^{\mu}[f] - \log\Exp^{\nu}[\exp(f)]\right\rbrace \in \R\cup\infty,
\end{equation*}
where $\mathcal{E}$ is the set of all bounded functions. If the $\sigma$-algebra is clear from the context, we omit the $\Hh$ and write $R(\mu||\nu)$. If $\Omega$ is Polish and $\Hh = \B(\Omega)$, then we only need to take the supremum over the set of all continuous bounded functions. 
\end{definition}
Let  $P\in \M(\T)$ be the following law governing a Markov-Feller diffusion process on $\T$.  $P$ is stipulated to be a weak solution (with respect to the canonical filtration) of the local martingale problem with infinitesimal generator
\begin{equation*}
\mathcal{L}_t(f) = \frac{1}{2} \sum_{1\leq j,k\leq d} a^{jk}(t,x)\frac{\partial^2 f}{\partial x^j x^k} + \sum_{1\leq j\leq d} b^j(t,x) \frac{\partial f}{\partial x^j},
\end{equation*}
for $f(x)$ in $C^2(\R^{d})$, i.e. the space of twice continuously differentiable functions. The initial condition (governing $P_0$) is $\mu_I \in \M(\R^d)$. The coefficients $a^{jk},b^j : [0,T]\times\R^d \to \R$ are assumed to be continuous (over $[0,T]\times\R^d$), and the matrix $a(t,x)$ is strictly positive definite for all $t$ and $x$. $P$ is assumed to be the unique weak solution. We note that the above infinitesimal generator is the same as in \cite[p 269]{dawson-gartner:87}  (note particularly Remark 4.4 in this paper). 

Our major result is the following. Let $\mu\in\M(\T)$ govern a random variable $X \in \T$. For some $x\in\T$, we note $\mu_{|[0,s],x}$, the regular conditional probability (rcp) given $X_r = x_r$ for all $r\in [0,s]$. The marginal of $\mu_{|[0,s],x}$ at some time $t\geq s$ is noted $\mu_{t | [0,s],x}$.
\begin{theorem}\label{thm:ratefunctionfinal}
For $\mu\in\M(\T)$, 
\begin{multline}
R\left(\mu ||P\right) =  R_{\F_0}(\mu || P) + \sup_{\sigma\in\J_*}E^{\mu(x)}\left[ \int_{0}^{T}\sup_{f\in\D}\bigg\lbrace \frac{\partial}{\partial t}E^{\mu_{t | [0,\underline{\sigma}(t)],x}}\left[f\right]\right. \\ \left.\left.-E^{\mu_{t|[0,\underline{\sigma}(t)],x}(y)} \left[\mathcal{L}_t f(y) + \frac{1}{2}\sum_{j,k=1}^d a^{jk}(t,y)\frac{\partial f}{\partial y^j}\frac{\partial f}{\partial y^k}\right]\right\rbrace dt\right].\label{eq:bigresult}
\end{multline}
Here $\D$ is the Schwartz Space of compactly supported functions $\R^{d} \to \R$, possessing continuous derivatives of all orders. If $\frac{\partial}{\partial t}E^{\mu_{t | [0,\underline{\sigma}(t)],x}}\left[f\right]$ does not exist, then we consider it to be $\infty$. 
\end{theorem}
Our paper has the following format. In Section 3 we make some preliminary definitions, defining the process $P$ against which the relative entropy is taken in this paper. In Section \ref{Sect:three} we employ the projective limits approach of \cite{dawson-gartner:87} to obtain the chief result of this paper: Theorem \ref{thm:ratefunctionfinal}. This gives an explicit integral representation of the relative entropy. In Section 5 we apply the result in Theorem  \ref{thm:ratefunctionfinal} to various corollaries, including the particular case when $\mu$ is the solution of a Martingale Problem. We finish by comparing our results to those of  \cite{budhiraja-dupuis-etal:12} and \cite{fischer:12}. 

\section{Preliminaries}
We outline some necessary definitions.

For $\sigma\in\J$ of the form $\sigma = \lbrace t_1,t_2,\ldots,t_m \rbrace$, let $\sigma_{;j} = \lbrace t_1,\ldots,t_j \rbrace$. We denote the number of elements in a partition $\sigma$ by $m(\sigma)$. We let $\J_s$ be the set of all partitions lying in $[0,s]$. For $0 < s < t \leq T$, we let $J_{s;t}$ be the set of all partitions of the form $\sigma\cup t$, where $\sigma\in \J_s$. Let $|\sigma| = \sup_{0\leq j \leq m(\sigma)-1} (t_{j+1} - t_j)$ and, for $t\in [0,T]$, let $\underline{\sigma}(t) = \sup\lbrace s\in\sigma: s\leq t\rbrace$.

Let $\pi_\sigma : \T \to \T_\sigma := \R^{d\times m(\sigma)}$ be the natural projection, i.e. such that $\pi_\sigma(x) = (x_{t_1},\ldots,x_{t_{m(\sigma)}})$. We similarly define the natural projection $\pi_{\alpha\gamma}:\T_\gamma\to\T_\alpha$ (for $\alpha\subseteq\gamma \in \J$), and we define $\pi_{[s,t]}: \T \to C([s,t];\R^d)$ to be the natural restriction of $x\in \T$ to $[s,t]$.  The expectation of some measurable function $f$ with respect to a measure $\mu$ is written as $\Exp^{\mu(x)}[f(x)]$, or simply $\Exp^{\mu}[f]$ when the context is clear. 

 For $s<t$, we write $\F_{s,t} = \pi_{[s,t]}^{-1}\B(C([s,t];\R^d))$ and $\F_\sigma = \pi_\sigma^{-1}\B(\T_\sigma)$. We define $\F_{s;t}$ to be the $\sigma$-algebra generated by $\F_s$ and $\F_{\gamma}$ (where $\gamma = [t]$).  For $\mu\in\M(\T)$, we denote its image laws by $\mu_{\sigma} := \mu\circ \pi_{\sigma}^{-1}\in\M(\T_\sigma)$ and $\mu_{[s,t]} := \mu\circ \pi_{[s,t]}^{-1} \in \M(C([s,t];\R^d))$. Let $(X_t)$ be a continuous stochastic process, adapted to $(\F_t)$ and governed by $\mu\in \M(\T)$. For $z\in \R^d$, we write the rcp given $X_s = z$ by $\mu_{|s,z}$.  For $x\in\C([0,s];\R^d)$ or $\T$, the rcp given that $X_u = x_u$ for all $0\leq u \leq s$ is written as $\mu_{|[0,s],x}$.  The rcp given that $X_u = x_u$ for all $u\leq s$, and $X_t = z$, is written as $\mu_{|s,x;t,z}$. For $\sigma\in\J_s$ and $z\in (\R^d)^{m(\sigma)}$, the rcp given that $X_u = z_u$ for all $u\in\sigma$ is written as $\mu_{|\sigma,z}$. All of these measures are considered to be in $\M(C([s,T](\R^d))$ (unless indicated otherwise in particular circumstances). The probability laws governing $X_t$ (for $t\geq s$), for each of these, are respectively $\mu_{t|s,z}$, $\mu_{t|[0,s],x}$ and $\mu_{t|\sigma,z}$. We clearly have $\mu_{s|s,z} = \delta_z$, for $\mu_s$ a.e. $z$, and similarly for the others.

\begin{remark}
See \cite[Definition 5.3.16]{karatzas-shreve:91} for a definition of a rcp. Technically, if we let $\mu^*_{|s,z}$ be the r.c.p given $X_s = z$ according to this definition, then $\mu_{|s,z} = \mu^*_{s,z}\circ\pi_{[s,T]}^{-1}$ and $\mu_{t|s,z} = \mu^*_{s,z}\circ\pi_t^{-1}$. By \cite[Theorem 3.18]{karatzas-shreve:91}, $\mu_{|s,z}$ is well-defined for $\mu_s$ a.e. $z$. Similar comments apply to the other rcp's defined above.
\end{remark}

In the definition of the relative entropy, we abbreviate $R_{\F_\sigma}(\mu || P)$ by $R_\sigma(R || P)$. If $\sigma =\lbrace t\rbrace$, we write $R_t(\mu || P)$. 

\section{The Relative Entropy $R(\cdot ||P)$ using Projective Limits}\label{Sect:three}
In this section we derive an integral representation of the relative entropy $R(\mu ||P)$, for arbitrary $\mu\in\M(\T)$. We use Sanov's Theorem to obtain the initial expressions in Theorem \ref{thm:DG}, before adapting the projective limits approach of \cite{dawson-gartner:87}  to obtain the central result (Theorem \ref{thm:ratefunctionfinal}).

We begin with a standard decomposition result for the relative entropy \cite{donsker-varadhan:83}.
 \begin{lemma}\label{lem:DonskVara}
Let $X$ be a Polish Space with sub $\sigma$-algebras $\G\subseteq\F \subseteq \B(X)$. Let $\mu$ and $\nu$ be probability measures on $(X,\F)$, and their regular conditional probabilities  over $\G$ be (respectively) $\mu_\omega$ and $\nu_\omega$. Then
\begin{equation*}
R_{\F}(\mu||\nu) = R_\G(\mu||\nu) + \Exp^{\mu(\omega)}\left[R_{\F}(\mu_\omega||\nu_\omega)\right].
\end{equation*}
\end{lemma}
The following Theorem is proved in Section \ref{sect:DGproof} of the Appendix.
\begin{theorem}\label{thm:DG}
If $\alpha,\sigma \in \J$ and $\alpha\subseteq\sigma$, then $R_\alpha(\mu||P) \leq R_\sigma(\mu||P)$. Furthermore,
\begin{align}
R_{\F_{s,t}}(\mu||P) &= \sup_{\sigma\in\J\cap [s,t]} R_{\sigma} (\mu||P),\label{eq:projlimit1}\\
R_{\F_{s;t}}(\mu||P) &= \sup_{\sigma\in\J_{s;t}} R_{\sigma} (\mu||P).\label{eq:projlimit2}
\end{align}
It suffices for the supremums in \eqref{eq:projlimit1} to take $\sigma\subset\mathcal{Q}_{s,t}$, where $\mathcal{Q}_{s,t}$ is any countable dense subset of $[s,t]$. Thus we may assume that there exists a sequence $\sigma^{(n)}\subset\mathcal{Q}_{s,t}$ of partitions such that $\sigma^{(n)}\subseteq \sigma^{(n+1)}$, $|\sigma^{(n)}| \to 0$ as $n\to \infty$ and
\begin{equation}
R_{\F_{s,t}}(\mu||P) = \lim_{n\to\infty} R_{\sigma^{(n)}}(\mu||P).\label{eq:projlimit3}
\end{equation}
\end{theorem}

We start with a technical lemma.
\begin{lemma}\label{lem:technical1}
Let $t>s$, $\alpha,\sigma \in \J_s$, $\sigma\subset \alpha$ and $s\in \sigma$. Then for $\mu_\sigma$ a.e. $x$, 
\begin{equation*}
R_t (\mu_{|\sigma,x}||P_{|s,x_s}) = R (\mu_{t|\sigma,x}||P_{t|s,x_s}).
\end{equation*}
Secondly,
\begin{equation*}
E^{\mu_{\sigma}(x)}\left[R_t (\mu_{|\sigma,x}||P_{|s,x_s})\right] \leq E^{\mu_{\alpha}(z)}\left[R_t (\mu_{|\alpha,z}||P_{|s,z_s})\right]. 
\end{equation*}
\end{lemma}
\begin{proof}
The first statement is immediate from Definition \ref{eq:relativeentropy} and the Markovian nature of $P$. For the second statement, it suffices to prove this in the case that $\alpha = \sigma \cup u$, for some $u < s$.  We note that, using a property of regular conditional probabilities, for $\mu_\sigma$ a.e $x$,
\begin{equation}
\mu_{t|\sigma,x} = E^{\mu_{u|\sigma,x}(\omega)}\left[ \mu_{t|\alpha,v(x,\omega)}\right],\label{eq:rcpd1}
\end{equation}
where $v(x,\omega)\in \T_{\alpha}$, $v(x,\omega)_u = \omega$, $v(x,\omega)_r = x_r$ for all $r\in\sigma$.

We consider $\mathfrak{A}$ to be the set of all finite disjoint partitions $\mathfrak{a}\subset \B(\R^d)$ of $\R^d$. The expression for the entropy in \cite[Lemma 1.4.3]{dupuis-ellis:97} yields
\begin{equation*}
E^{\mu_\sigma(x)}\left[ R \left(\mu_{t|\sigma,x}||P_{t|s,x_s}\right)\right] = E^{\mu_\sigma(x)}\left[ \sup_{\mathfrak{a}\in\mathfrak{A}}\sum_{A\in\mathfrak{a}}\mu_{t|\sigma,x}(A)\log\frac{\mu_{t|\sigma,x}(A)}{P_{t|s,x_s}(A)}\right] .
\end{equation*}
Here the summand is considered to be zero if $\mu_{t|\sigma,x}(A) = 0$, and infinite if $\mu_{t|\sigma,x}(A) > 0$ and $P_{t|s,x_s}(A)=0$. Making use of \eqref{eq:rcpd1}, we find that
\begin{align*}
& E^{\mu_\sigma(x)}\left[R \left(\mu_{t|\sigma,x}||P_{t|s,x_s}\right)\right] \\ &= E^{\mu_\sigma(x)}\left[\sup_{\mathfrak{a}\in\mathfrak{A}}\sum_{A\in\mathfrak{a}}E^{\mu_{u|\sigma,x}(w)}\left[\mu_{t|\alpha,v(x,w)}(A)\right]\log\frac{\mu_{t|\sigma,x}(A)}{P_{t|s,x_s}(A)}\right] \\
&\leq E^{\mu_\sigma(x)} E^{\mu_{u|\sigma,x}(\omega)}\left[ \sup_{\mathfrak{a}\in\mathfrak{A}}\sum_{A\in\mathfrak{a}}\mu_{t|\alpha,v(x,\omega)}(A)\log\frac{\mu_{t|\sigma,x}(A)}{P_{t|s,x_s}(A)}\right]\\
&= E^{\mu_\alpha(z)}\left[\sup_{\mathfrak{a}\in\mathfrak{A}}\sum_{A\in\mathfrak{a}}\mu_{t|\alpha,z}(A)\log\frac{\mu_{t|\sigma,\pi_{\sigma\alpha} z}(A)}{P_{t|s,z_s}(A)}\right].
\end{align*}
We note that, for $\mu_{\alpha}$ a.e. $z$, if $\mu_{t|\sigma,\pi_{\sigma\alpha} z}(A) = 0$ in this last expression, then $\mu_{t|\alpha,z}(A) =0$ and we consider the summand to be zero. To complete the proof of the lemma, it is thus sufficient to prove that for $\mu_\alpha$ a.e. $z$
\begin{equation*}
\sup_{\mathfrak{a}\in\mathfrak{A}}\sum_{A\in\mathfrak{a}}\mu_{t|\alpha,z}(A)\log\frac{\mu_{t|\alpha,z}(A)}{P_{t|s,z_s}(A)} \geq \sup_{\mathfrak{a}\in\mathfrak{A}}\sum_{A\in\mathfrak{a}}\mu_{t|\alpha,z}(A)\log\frac{\mu_{t|\sigma,\pi_{\sigma\alpha} z}(A)}{P_{t|s,z_s}(A)} .
\end{equation*}
But, in turn, the above inequality will be true if we can prove that for each partition $\mathfrak{a}$ such that $P_{t|s,z_s}(A) > 0$ and $\mu_{t|\sigma,\pi_{\sigma\alpha} z}(A) > 0$ for all $A\in\mathfrak{a}$,
\begin{equation*}
\sum_{A\in\mathfrak{a}}\mu_{t|\alpha,z}(A)\log\frac{\mu_{t|\alpha,z}(A)}{P_{t|s,z_s}(A)} -\sum_{A\in\mathfrak{a}}\mu_{t|\alpha,z}(A)\log\frac{\mu_{t|\sigma,\pi_{\sigma\alpha} z}(A)}{P_{t|s,z_s}(A)}\geq 0. 
\end{equation*}
The left hand side is equal to $\sum_{A\in\mathfrak{a}}\mu_{t|\alpha,z}(A)\log\frac{\mu_{t|\alpha,z}(A)}{\mu_{t|\sigma,\pi_{\sigma\alpha} z}(A)}$. An application of Jensen's inequality demonstrates that this is greater than or equal to zero.
\end{proof}
\begin{remark}
If, contrary to the definition, we briefly consider $\mu_{|[0,t],x}$ to be a probability measure on $\T$, such that $\mu(A)=1$ where $A$ is the set of all points $y$ such that $y_s = x_s$ for all $s\leq t$, then it may be seen from the definition of $R$ that
\begin{equation}
R_{\F_T}\left(\mu_{|[0,t],x}||P_{|[0,t],x}\right) = R_{\F_{t,T}}\left(\mu_{|[0,t],x}||P_{|[0,t],x}\right) = R_{\F_{t,T}}\left(\mu_{|[0,t],x}||P_{|t,x_t}\right).  \label{eq:I2smallident}
\end{equation}
We have also made use of the Markov Property of $P$. This is why our convention, to which we now return, is to consider $\mu_{|[0,t],x}$ to be a probability measure on $(C([t,T];\R^d),\F_{t,T})$.
\end{remark}

This leads us to two alternative expressions for $R(\mu || P)$.
\begin{lemma}\label{lem:part1}
Each $\sigma$ in the supremums below is of the form $\lbrace t_1 < t_2 < \ldots < t_{m(\sigma)-1} < t_{m(\sigma)}\rbrace$ for some integer $m(\sigma)$. 
\begin{align}
&R\left(\mu||P\right) = R_0(\mu || P) + \sum_{j=1}^{m(\sigma)-1}E^{\mu_{[0,t_j]}(x)}\left[R_{\F_{t_j,t_{j+1}}}\left( \mu_{| [0,t_j],x}||P_{|t_j,x_{t_j}}\right)\right], \label{eq:I2ident1}\\
&R\left(\mu||P\right) =  R_0(\mu || P) +  \sup_{\sigma\in\J_*}\sum_{j=1}^{m(\sigma)-1}E^{\mu_{\sigma_{;j}}(x)}\left[ R_{t_{j+1}}\left( \mu_{t_{j+1} | \sigma_{;j},x}||P_{t_{j+1}|t_j,x_{t_j}}\right)\right], \label{eq:I2ident2} \\
&E^{\mu_{[0,s]}(x)}\left[R_t\left(\mu_{t|[0,s],x}||P_{t|s,x_s}\right)\right] = \sup_{\sigma\in\J_s}E^{\mu_{\sigma}(y)}\left[ R_t\left(\mu_{t|\sigma,y}||P_{t|s,y_s}\right)\right],\label{eq:I2ident4}
\end{align}
where in this last expression $0\leq s < t \leq T$. 
\end{lemma}
\begin{proof}
Consider the sub $\sigma$-algebra $\F_{0,t_{m(\sigma)-1}}$. We then find, through an application of Lemma \ref{lem:DonskVara} and \eqref{eq:I2smallident}, that
\begin{multline*}
R\left(\mu ||P\right) = R_{\F_{0,t_{m(\sigma)-1}}} \left(\mu ||P\right)+ \\ E^{\mu_{[0,t_{m(\sigma)-1}]}(x)}\left[ R_{\F_{t_{m(\sigma)-1},t_{m(\sigma)}}}\left( \mu_{| [0,t_{m(\sigma)-1}],x}||P_{|t_{m(\sigma)-1},x_{t_{m(\sigma)-1}}}\right)\right].\end{multline*}
We may continue inductively to obtain the first identity. 

We use Theorem \ref{thm:DG} to prove the second identity. It suffices to take the supremum over $\J_*$, because $R_\sigma(\mu ||P) \geq R_\gamma(\mu ||P)$ if $\gamma \subset \sigma$. It thus suffices to prove that
\begin{equation}
R_{\sigma}(\mu||P) = R_0(\mu || P) +\sum_{j=1}^{m(\sigma)-1}E^{\mu_{\sigma_{;j}}(x)}\left[ R_{t_{j+1}}\left( \mu_{t_{j+1} | \sigma_{;j},x}||P_{t_{j+1}|t_j,x_{t_j}}\right)\right].
\end{equation}
But this also follows from repeated application of Lemma \ref{lem:DonskVara}. To prove the third identity, we firstly note that 
\begin{align*}
R_{\F_{s;t}}\left(\mu||P\right) &=  R_0(\mu || P) +  \sup_{\sigma\in\J_{s;t}}\sum_{j=1}^{m(\sigma)-1}E^{\mu_{\sigma_{;j}}(x)}\left[ R_{t_{j+1}}\left( \mu_{t_{j+1} | \sigma_{;j},x}||P_{t_{j+1}|t_j,x_{t_j}}\right)\right].\\
&= \sup_{\sigma\in\J_s}\left\lbrace R_{\sigma}(\mu||P) + E^{\mu_{\sigma}(x)}\left[R_t\left(\mu_{t|\sigma,x}||P_{t|s,x_s}\right)\right]\right\rbrace.
\end{align*}
The proof of this is entirely analogous to that of the second identity, except that it makes use of \eqref{eq:projlimit2} instead of \eqref{eq:projlimit1}. But, after another application of Lemma \ref{lem:DonskVara}, we also have that
\begin{equation*}
R_{\F_{s;t}}\left(\mu ||P\right)  = R_{\F_{0,s}}(\mu ||P) + E^{\mu_{[0,s]}(x)}\left[ R_t(\mu_{t|[0,s],x} ||P_{t|s,x_s})\right].
\end{equation*}
On equating these two different expressions for $R_{\F_{s;t}}\left(\mu ||P\right)$, we obtain
\begin{multline*}
E^{\mu_{[0,s]}(x)}\left[R_t(\mu_{t|[0,s],x} ||P_{t|s,x_s})\right] = \sup_{\sigma\in\J_{s}}\left\lbrace \left(R_{\sigma}(\mu ||P) - R_{\F_{0,s}}(\mu ||P)\right) \right. \\ \left.+ E^{\mu_{\sigma}(x)}\left[R_t\left(\mu_{t|\sigma,x} ||P_{t|s,x_s}\right)\right]\right\rbrace.
\end{multline*}
Let $(\sigma^{(k)}) \subset \J_s$, $\sigma^{(k-1)}\subseteq \sigma^{(k)}$ be such that $\lim_{k\to\infty}R_{\sigma^{(k)}}(\mu ||P) = R_{\F_{0,s}}(\mu ||P) $. Such a sequence exists by \eqref{eq:projlimit1}. Similarly, let $(\gamma^{(k)})\subseteq \J_s$,  be a sequence such that  $E^{\mu_{\gamma^{(k)}}(x)}\left[R_t\left(\mu_{t|\gamma^{(k)},x} ||P_{t|s,x_s}\right)\right]$ is strictly nondecreasing and asymptotes  as $k\to\infty$ to $\sup_{\sigma\in\J_s}E^{\mu_{\sigma}(x)}\left[R_t\left(\mu_{t|\sigma,x} ||P_{t|s,x_s}\right)\right]$. Lemma \ref{lem:technical1} dictates that \begin{equation*}
E^{\mu_{\sigma^{(k)}\cup\gamma^{(k)}}(x)}\left[R_t\left(\mu_{t|\sigma^{(k)}\cup\gamma^{(k)},x} ||P_{t|s,x_s}\right)\right]
\end{equation*}
asymptotes to the same limit as well. Clearly $\lim_{k\to\infty}R_{\sigma^{(k)}\cup\gamma^{(k)}}(\mu ||P) = R_{\F_{0,s}}(\mu ||P) $ because of the identity at the start of Theorem \ref{thm:DG}. This yields the third identity.
\end{proof}
\subsection{Proof of Theorem \ref{thm:ratefunctionfinal}}

In this section we work towards the proof of Theorem \ref{thm:ratefunctionfinal}, making use of some results in \cite{dawson-gartner:87}. However first we require some more definitions.

If $K\subset \R^{d}$ is compact, let $\D_K$ be the set of all $f\in\D$ whose support is contained in $K$. The corresponding space of real distributions is $\D^{'}$, and we denote the action of $\theta\in\D^{'}$ by $\langle \theta,f\rangle$.  If $\theta\in\M(\R^d)$, then clearly $\langle \theta,f\rangle = \Exp^{\theta}[f]$.  We let $C^{2,1}_0(\R^d)$ denote the set of all continuous functions, possessing continuous spatial derivatives of first and second order, a continuous time derivative of first order, and of compact support.  For $f\in\D$ and $t\in[0,T]$, we define the random variable $\nabla_t f: \R^d\to \R^d$ such that $(\nabla_t f(y))^i= \sum_{j=1}^d a^{ij}(t,y)\frac{\partial f}{\partial y^j}$ (for $x\in\T$, we may also understand $\nabla_t f(x) := \nabla_t f(x_t)$). Let $a_{ij}$ be the components of the matrix inverse of $a^{ij}$. For random variables $X,Y:\T\to\R^d$, we define the inner product $(X,Y)_{t,x} = \sum_{i,j=1}^d X^i(x) Y^j(x) a_{ij}(t,x_t)$, with associated norm $|X|_{t,x}^2 = (X(x),X(x))_{t,x}^2$. We note that $|\nabla_t f|_{t,x}^2 = \sum_{i,j=1}^d a^{ij}(t,x_t)\frac{\partial f}{\partial z^i}(x_t)\frac{\partial f}{\partial z^j}(x_t)$.

Let $\MM$ be the space of all continuous maps $[0,T] \to \M(\R^d)$, equipped with the topology of uniform convergence. For $s\in[0,T]$, $\vartheta\in \MM$ and $\nu\in\M(\R^d)$ we define $\n(s,\vartheta,\nu) \geq 0$ and such that
\begin{equation}
\n(s,\vartheta,\nu)^2 = \sup_{f\in\D}\left\lbrace\langle\vartheta,f\rangle -\frac{1}{2} E^{\nu(y)}\left[\left|\nabla_t f\right|_{t,y}^2\right] \right\rbrace.\label{defn:n}
\end{equation}
This definition is taken from \cite[Eq. (4.7)]{dawson-gartner:87} - we note that $\n$ is convex in $\vartheta$. For $\gamma\in \M(\T)$, we may naturally write $\n(s,\gamma,\nu) :=\n(s,\omega,\nu)$, where $\omega$ is the projection of $\gamma$ onto $\MM$, i.e. $\omega(s) = \gamma_s$. It is shown in \cite{dawson-gartner:87} that this projection is continuous. 
The following two definitions, lemma and two propositions are all taken (with some small modifications) from \cite{dawson-gartner:87}.
\begin{definition}
Let I be an interval of the real line. A measure $\mu\in \M(\T)$ is called absolutely continuous if for each compact set $K\subset \R^d$ there exists a neighborhood $U$ of 0 in K and an absolutely continuous function $H_K :I \to \R$ such that
\begin{equation*}
\left|  \Exp^{\mu_u }[f] - \Exp^{\mu_v}[f] \right| \leq \left| H_K(u) - H_K(v)\right|,
\end{equation*}
for all $u,v\in I$ and $f\in U_K$.
\end{definition}
\begin{lemma}\label{lem:abscontderiv}
\cite[Lemma 4.2]{dawson-gartner:87} If $\mu$ is absolutely continuous over an interval $I$, then its derivative exists (in the distributional sense) for Lebesgue a.e. $t\in I$. That is, for Lebesure a.e. $t\in I$, there exists $\dot{\mu}_t \in \D^{'}$ such that for all $f\in \D$
\begin{equation*}
\lim_{h\to 0}\frac{1}{h}\left( \langle \mu_{t+h},f\rangle - \langle \mu_t,f\rangle\right) = \langle \dot{\mu}_t,f\rangle.
\end{equation*}
\end{lemma}
\begin{definition}\label{defn:HilbertSpace1}
For $\nu\in\M(C([s,t];\R^d))$, and $0\leq s < t \leq T$, let $L^{2}_{s,t}(\nu)$ be the Hilbert space of all measurable maps $h: [s,t]\times \R^{d} \to \R^{d}$ with inner product
\begin{equation*}
[ h_1,h_2 ] = \int_s^t\Exp^{\nu_u(x)}\left[(h_1(u,x),h_2(u,x))_{u,x}\right] du.
\end{equation*}
We denote by $L^{2}_{s,t,\nabla}(\nu)$ the closure in $L_{s,t}^{2}(\nu)$ of the linear subset generated by maps of the form $(x,u)\to \nabla_u f$, where $f\in C^{2,1}_0([s,t],\R^{d})$. We note that functions in $L^{2}_{s,t,\nabla}(\nu)$ only need to be defined $du\otimes\nu_u(dx)$ almost everywhere.
\end{definition}
Recall that $\n$ is defined in \eqref{defn:n}, and note that $\langle {}^*\mathcal{L}_t\mu_t,f\rangle := \langle \mu_t,\mathcal{L}_t f\rangle$. 
 \begin{proposition}\label{prop:intermediate}
Assume that $\mu\in\M(C([r,s];\R^d))$, such that $\mu_r = \delta_y$ for some $y\in \R^d$ and $0\leq r<s\leq T$.  We have that \cite[Eq. 4.9 and Lemma 4.8]{dawson-gartner:87}
 \begin{multline}
\int_r^s \n(t,\dot{\mu}_t-{}^*\mathcal{L}_t\mu_t,\mu_t)^2 dt = \sup_{f\in C^{2,1}_0(\R^d)}\left\lbrace \Exp^{\mu_s(x)}[f(s,x)] -f(r,y) \right. \\ \left.- \int^s_r \Exp^{\mu_t(x)}\left[\left(\frac{\partial}{\partial t} + \mathcal{L}_t\right)f(t,x) + \frac{1}{2}\left|\nabla_t f(t,x)\right|^2_{t,x}\right] dt\right\rbrace .\label{eq:DG3}
 \end{multline}
It clearly suffices to take the supremum over a countable dense subset. Assume now that $\int_r^s \n(t,\dot{\mu}_t-{}^*\mathcal{L}_t\mu_t,\mu_t)^2 dt < \infty$. Then for Lebesgue a.e. $t$, $\dot{\mu}_t = {}^*\mathcal{K}_t\mu_t$, where \cite[Lemma 4.8(3)]{dawson-gartner:87}
\begin{equation}\label{eq:KidentityDG}
\mathcal{K}_t f(\cdot) = \mathcal{L}_t  f(\cdot)+ \sum_{1\leq j\leq d}(h^{\mu}(t,\cdot))^j\frac{\partial f}{\partial x^j}(\cdot),
\end{equation}
for some $h^{\mu}\in L^{2}_{r,s,\nabla}(\mu)$ which satisfies  \cite[Lemma 4.8(4)]{dawson-gartner:87}
\begin{equation}
\int_r^s \n(t,\dot{\mu}_t-{}^*\mathcal{L}_t\mu_t,\mu_t)^2 dt = \frac{1}{2}\int_r^s \Exp^{\mu_t(x)}\left[\left| h^\mu(t,x)\right|_{t,x}^2\right] dt < \infty.\label{eq:hmuratefunction}
\end{equation}
\end{proposition}
 \begin{remark}
We reach \eqref{eq:DG1} from the proof of Lemma 9 in \cite[Eq 4.10]{dawson-gartner:87}. One should note also that \cite{dawson-gartner:87} write the relative entropy $R$ as $L^{(1)}_\nu$ in (their) equation (4.10). To reach \eqref{eq:DG2}, we also use the equivalence between (4.7) and (4.8) in \cite{dawson-gartner:87}.
 \end{remark}
\begin{proposition}\label{theorem:DGbig}
Assume that $\mu\in\M(\T)$, such that $\mu_{r} = \delta_y$ for some $y\in \R^d$ and $0\leq r<s\leq T$. If $R_{\F_{r,s}}(\mu ||P_{|r,y}) < \infty$, then $\mu$ is absolutely continuous on $[r,s]$, and \cite[Lemma 4.9]{dawson-gartner:87}
\begin{equation}
R_{\F_{r,s}}(\mu ||P_{|r,y}) \geq \int_r^s \n(t,\dot{\mu}_{t}-{}^*\mathcal{L}_t\mu_{t},\mu_{t})^2 dt.\label{eq:DG1}
\end{equation}
Here the derivative $\dot{\mu}_{t}$ is defined in Lemma \ref{lem:abscontderiv}. For all $f\in\D$,  \cite[Eq (4.35)]{dawson-gartner:87}
\begin{equation}
\Exp^{\mu_{s}}[f] -\log\Exp^{P_{s|r,y}}\left[\exp(f)\right] \leq \int_r^s \n(t,\dot{\mu}_{t}-{}^*\mathcal{L}_t\mu_t,\mu_t)^2 dt.\label{eq:DG2}
 \end{equation}
 \end{proposition}
We are now ready to prove Theorem \ref{thm:ratefunctionfinal} (the central result).
\begin{proof}
Fix a partition $\sigma = \lbrace t_1,\ldots,t_{m}\rbrace$. We may conclude from \eqref{eq:I2ident1} and \eqref{eq:DG1} that, 
\begin{multline}
R\left(\mu || P\right) \geq R_0\left(\mu || P\right) + \\ \sum_{j=1}^{m-1}E^{\mu_{[0,t_j]}(x)}\int_{t_j}^{t_{j+1}} \n(t,\dot{\mu}_{t | [0,t_j],x}-{}^*\mathcal{L}_t\mu_{t|[0,t_j],x},\mu_{t|[0,t_j],x})^2 dt.
\end{multline}
The integrand on the right hand side is measurable with respect to $E^{\mu_{[0,t_j]}(x)}$ due to the equivalent expression \eqref{eq:DG3}. We may infer from \eqref{eq:DG2} that,
\begin{multline}
E^{\mu_{[0,t_j]}(x)}\int_{t_j}^{t_{j+1}} \n(t,\dot{\mu}_{t | [0,t_j],x}-{}^*\mathcal{L}_t\mu_{t|[0,t_j],x},\mu_{t|t_j,x})^2 dt \\ \geq \Exp^{\mu_{[0,t_j]}(x)}\left[ \sup_{f\in\D}\left\lbrace  \Exp^{\mu_{t_{j+1}|[0,t_j],x}}[f] -\log\Exp^{P_{t_{j+1}|t_j,x_{t_j}}}\left[\exp(f)\right]\right\rbrace\right]\\
\\ = \Exp^{\mu_{[0,t_j]}(x)}\left[ \sup_{f\in C_b(\R^d)}\left\lbrace \Exp^{\mu_{t_{j+1}|[0,t_j],x}}[f] -\log\Exp^{P_{t_{j+1}|t_j,x_{t_j}}}\left[\exp(f)\right]\right\rbrace\right].\label{eq:tmptmp}
\end{multline}
This last step follows by noting that if $\nu\in\M(\R^d)$, and $f\in C_b(\R^d)$, and the expectation of $f$ with respect to $\nu$ is finite, then there exists a series $(K_n)\subset \R^d$ of compact sets such that 
\begin{equation*}
\int_{\R^d}f(x)d\nu(x) = \lim_{n\to\infty}\int_{K_n}f(x)d\nu(x).
\end{equation*}
In turn, for each $n$ there exist $(f_n^{(m)})\in \D_{K_n}$ such that we may write
\begin{equation*}
\int_{K_n}f(x)d\nu(x) = \lim_{m\to\infty}\int_{K_n}f_n^{(m)}(x)d\nu(x).
\end{equation*}
This allows us to conclude that the two supremums are the same. The last expression in \eqref{eq:tmptmp} is merely
\begin{equation*}
E^{\mu_{[0,t_j]}(x)}\left[R_{t_{j+1}}\left( \mu_{t_{j+1} | [0,t_j],x}||P_{t_{j+1}|t_j,x_{t_j}}\right)\right].
\end{equation*}
By \eqref{eq:I2ident4}, this is greater than or equal to
\begin{equation*}
E^{\mu_{\sigma_{;j}}(y)}\left[R_{t_{j+1}}\left(\mu_{t_{j+1}|\sigma_{;j},y}||P_{t_{j+1}|t_j,y_{t_j}}\right)\right].
\end{equation*}
We thus obtain the theorem using \eqref{eq:I2ident2}. 
\end{proof}
\section{Some Corollaries}\label{Sect:Particular}
We state some corollaries of Theorem \ref{thm:ratefunctionfinal}. In the course of this section we make progressively stronger assumptions on the nature of $\mu$, culminating in the elegant expression for $R(\mu || P)$ when $\mu$ is a solution of a martingale problem. We finish by comparing our work with that of \cite{budhiraja-dupuis-etal:12,fischer:12}.
\begin{corollary}
Suppose that $\mu\in\M(\T)$ and $R(\mu || P) < \infty$. Then for all $s$ and $\mu$ a.e. $x$, $\mu_{|[0,s],x}$ is absolutely continuous over $[s,T]$. For each $s\in [0,T]$ and $\mu$ a.e. $x\in\T$, for Lebesgue a.e. $t\geq s$
\begin{equation}
\dot{\mu}_{t|[0,s],x} = {}^*\mathcal{K}^\mu_{t|s,x}\mu_{t|[0,s],x}\label{eq:temp10}
\end{equation}
where for some $h^{\mu}_{s,x} \in L^{2}_{s,T,\nabla}(\mu_{|[0,s],x})$
\begin{equation}
 \mathcal{K}^\mu_{t|s,x}f(y) = \mathcal{L}_t f(y) + \sum_{j=1}^d h^{\mu,j}_{s,x}(t,y)\frac{\partial f}{\partial y^j}(y).\label{eq:temp11}
\end{equation}
Furthermore,
\begin{equation}
R\left(\mu ||P\right) =  R_0(\mu || P) +\frac{1}{2} \sup_{\sigma\in\J_*}\int_{0}^{T}E^{\mu(w)}E^{\mu_{t|[0,\underline{\sigma}(t)],w}(z)}\left[ \left| h^\mu_{\underline{\sigma}(t),w}(t,z)\right|_{t,z}^2\right] dt.
\label{eq:temp7}
\end{equation}
For any dense countable subset $\mathcal{Q}_{0,T}$ of $[0,T]$, there exists a series of partitions $\sigma^{(n)}\subset\sigma^{(n+1)}\in\mathcal{Q}_{0,T}$, such that as $n\to\infty$, $|\sigma^{(n)}|\to 0$, and
\begin{equation}\label{eq:corollaryR}
R\left(\mu ||P\right) =  R_0(\mu || P) + \frac{1}{2}\lim_{n\to\infty}\int_{0}^{T}E^{\mu(w)}E^{\mu_{t|[0,\underline{\sigma}^{(n)}(t)],w}(z)}\left[ \left| h^\mu_{\underline{\sigma}^{(n)}(t),w}(t,z)\right|_{t,z}^2\right] dt.
\end{equation}
\end{corollary}
\begin{remark}
It is not immediately clear that we may simplify \eqref{eq:temp7} further (barring further assumptions). The reason for this is that  we only know that \newline $E^{\mu_{|[0,\underline{\sigma}(t)],w}(z)}\left[ \left| h^\mu_{\underline{\sigma}(t),w}(t,z)\right|_{t,z}^2\right]$ is measurable (as a function of $w$), but it has not been proven that $h^\mu_{\underline{\sigma}(t),w}(t,z)$ is measurable (as a function of $w$).
\end{remark}
\begin{proof}
Let $\sigma = \lbrace 0=t_1,\ldots,t_m=T\rbrace$ be an arbitrary partition. For all $j<m$, we find from Lemma \ref{lem:part1} that  $R_{\F_{t_j,t_{j+1}}}\left( \mu_{| [0,t_j],x} || P_{|t_j,x_{t_j}}\right) < \infty$ for $\mu_{[0,t_j]}$ a.e. $x\in C([0,t_j];\R^d)$. We thus find that, for all such $x$, $\mu_{|[0,t_j],x}$ is absolutely continuous on $[t_j,t_{j+1}]$ from Proposition \ref{theorem:DGbig}. We are then able to obtain \eqref{eq:temp10} and \eqref{eq:temp11} from Propositions \ref{prop:intermediate} and \ref{theorem:DGbig}.
From  \eqref{eq:hmuratefunction}, \eqref{eq:bigresult} and \eqref{eq:temp10} we find that
\begin{equation}
R\left(\mu || P\right) =  R_0(\mu || P) + \frac{1}{2}\sup_{\sigma\in\J_*}E^{\mu(x)}\int_{0}^{T}E^{\mu_{t|[0,\underline{\sigma}(t)],x}(z)}\left[\left| h^\mu_{\underline{\sigma}(t),x}(t,z)\right|_{t,z}^2\right] dt.
\label{eq:temp12}
\end{equation}
The above integral must be finite (since we are assuming $R(\mu || P)$ is finite). Furthermore $E^{\mu_{t|[0,\underline{\sigma}(t)],x}(z)}\left[\left| h^\mu_{\underline{\sigma}(t),x}(t,z)\right|_{t,z}^2\right]$ is $(t,x)$ measurable as a consequence of the equivalent form \eqref{eq:DG3}. This allows us to apply Fubini's Theorem to obtain \eqref{eq:temp7}. The last statement on the sequence of maximising partitions follows from Theorem \ref{thm:DG}.
\end{proof}
\begin{corollary}\label{cor:last}
Suppose that $R(\mu|| P) < \infty$. Suppose that for all $s\in\mathcal{Q}_{0,T}$ (any countable, dense subset of $[0,T]$), for $\mu$ a.e. $x$ and Lebesgue a.e. $t$, $h^{\mu}_{s,x}(t,x_t) = E^{\mu_{|[0,s],x;t,x_t}(w)}h^{\mu}(t,w)$ for some progressively-measurable random variable $h^{\mu}:[0,T]\times\T \to \R^d$. Then
\begin{equation*}
R\left(\mu ||P\right) =  R_0(\mu || P) +\frac{1}{2}\int_{0}^{T}E^{\mu(w)}\left[ \left| h^\mu(t,w)\right|_{t,w_t}^2\right] dt.
\end{equation*}
\end{corollary}
\begin{proof}
Let $\mathcal{G}^{s,x;t,y}$ be the sub $\sigma$-algebra consisting of all $B\in \mathcal{B}(\mathcal{T})$ such that for all $w\in B$, $w_r = x_r$ for all $r\leq s$ and $w_t = y$. Thus $h^{\mu}_{s,x}(t,y) = E^{\mu_{|[0,s],x;t,y}(w)}h^{\mu}(t,w) = E^{\mu}\left[h^{\mu}(t,\cdot)|\mathcal{G}^{s,x;t,y}\right]$. By \cite[Corollary 2.4]{revuz-yor:91}, since \newline$\cap_{s< t} \mathcal{G}^{s,x;t,x_t}= \mathcal{G}^{t,x;t,x_t}$ (restricting to $s\in\mathcal{Q}_{0,T}$), for $\mu$ a.e. $x$, 
\begin{equation}\label{eq:revuz}
\lim_{s\to t^{-}}E^{\mu_{|[0,s],x;t,x_t}(w)}h^{\mu}(t,w)= h^{\mu}(t,x),
\end{equation}
where $s\in\mathcal{Q}_{0,T}$. By the properties of the regular conditional probability, we find from \eqref{eq:corollaryR} that
\begin{multline}
R\left(\mu ||P\right) =  R_0(\mu || P) +\\ \frac{1}{2}\lim_{n\to\infty}\int_{0}^{T}E^{\mu(w)}\left[ \left| E^{\mu_{|[0,\underline{\sigma}^{(n)}(t)],w;t,w_t}(v)} \left[h^\mu(t,v) \right]\right|_{t,w_t}^2\right] dt.\label{eq:entropytmp4}
\end{multline}
By assumption, the above limit is finite. Thus by Fatou's Lemma, and using the properties of the regular conditional probability,
\begin{multline*}
R\left(\mu ||P\right) \geq R_0(\mu || P) + \\ \frac{1}{2}\int_{0}^{T}E^{\mu(w)}\left[ \linf{n}\left| E^{\mu_{|[0,\underline{\sigma}^{(n)}(t)],w;t,w_t}(v)} \left[h^\mu(t,v) \right]\right|_{t,w_t}^2\right] dt.
\end{multline*}
Through use of \eqref{eq:revuz}, 
\begin{equation*}
R\left(\mu ||P\right) \geq  R_0(\mu || P) +\frac{1}{2}\int_{0}^{T}E^{\mu(w)}\left[ \left| h^\mu(t,w)\right|_{t,w_t}^2\right] dt.
\end{equation*}
Conversely, through an application of Jensen's Inequality to \eqref{eq:entropytmp4}
\begin{equation*}
R\left(\mu ||P\right) \leq  R_0(\mu || P) + \frac{1}{2}\lim_{n\to\infty}\int_{0}^{T}E^{\mu(w)}\left[E^{\mu_{|[0,\underline{\sigma}^{(n)}(t)],w;t,w_t}(v)} \left[\left| h^\mu(t,v)\right|_{t,w_t}^2\right]\right] dt.
\end{equation*}
A property of the regular conditional probability yields
\begin{equation*}
R\left(\mu ||P\right) \leq  R_0(\mu || P) +\frac{1}{2}\int_{0}^{T}E^{\mu(w)}\left[ \left| h^\mu(t,w)\right|_{t,w_t}^2\right] dt.
\end{equation*}
\end{proof}
\begin{remark}
The condition in the above corollary is satisfied when $\mu$ is a solution to a Martingale Problem - see Lemma \ref{Thm:final}.
\end{remark}

We may further simplify the expression in Theorem \ref{thm:ratefunctionfinal} when $\mu$ is a solution to the following Martingale Problem. Let $\lbrace c^{jk},e^j\rbrace$ be  progressively-measurable functions $[0,T]\times \T \to\R$. We suppose that $c^{jk} = c^{kj}$. For all $1\leq j,k \leq d$, $c^{jk}(t,x)$ and $e^j(t,x)$ are assumed to be bounded for $x\in L$ (where $L$ is compact) and all $t\in [0,T]$. For $f\in C_0^{2}(\R^d)$ and $x\in\T$, let
\begin{equation*}
\mathcal{M}_u(f)(x) = \sum_{1\leq j,k\leq d}c^{jk}(u,x) \frac{\partial^2 f}{\partial y^j \partial y^k}(x_u) + \sum_{1\leq j\leq d}e^j(u,x)\frac{\partial f}{\partial y^j}(x_u).
\end{equation*}
We assume that for all such $f$, the following is a continuous martingale (relative to the canonical filtration) under $\mu$
\begin{equation}
f(X_t) - f(X_0) - \int^t_0\mathcal{M}_u f(X)du.\label{eq:mumartingale}
\end{equation}
The law governing $X_0$ is stipulated to be $\nu \in \M(\R^d)$. 

From now on we switch from our earlier convention and we consider $\mu_{|[0,s],x}$ to be a measure on $\T$ such that, for $\mu$ a.e. $x\in\T$, $\mu_{|[0,s],x}(A_{s,x}) = 1$ where $A_{s,x}$ is the set of all $X\in\T$ satisfying $X_t = x_t$ for all $0\leq t\leq s$. This is a property of a regular conditional probability (see \cite[Theorem 3.18]{karatzas-shreve:91}). Similarly, $\mu_{|s,x;t,y}$ is considered to be a measure on $\T$ such that for $\mu$ a.e. $x\in\T$, $\mu_{|s,x;t,y}(B_{s,x;t,y}) = 1$, where $B_{s,x;t,y}$ is the set of all $X\in A_{s,x}$ such that $X_t = y$. We may apply Fubini's Theorem (since $f$ is compactly supported and bounded) to \eqref{eq:mumartingale} to find that
\begin{equation*}
\langle \mu_{t|[0,s],x},f\rangle - f(x_s) = \int^t_s \Exp^{\mu_{|[0,s],x}}\left[\mathcal{M}_u f\right] du.
\end{equation*}
This ensures that $\mu_{|[0,s]}$ is absolutely continuous over $[s,T]$, and that
\begin{equation}
\langle \dot{\mu}_{t|[0,s],x},f\rangle =\Exp^{\mu_{|[0,s],x}}\left[\mathcal{M}_t f\right].\label{eq:Mderivative}
\end{equation}
\begin{lemma}\label{Thm:final}
If $R(\mu || P) < \infty$ then for Lebesgue a.e. $t\in[0,T]$ and $\mu$ a.e.  $x\in\T$,
\begin{equation}
a(t,x_t) =  c(t,x).\label{eq:aequalc}
\end{equation}
If $R(\mu || P) < \infty$ then
\begin{equation}
R(\mu||P) = R(\nu || \mu_I) +  \frac{1}{2}\Exp^{\mu(x)}\left[\int_0^T \left| b(s,x_s) -e(s,x)\right|_{s,x_s}^2ds\right].\label{eq:lastidentity}
\end{equation}
\end{lemma}
\begin{proof}
It follows from $R(\mu||P) < \infty$, \eqref{eq:temp10} and \eqref{eq:temp11} that for all $s$ and $\mu$ a.e. $x$, for Lebesgue a.e. $t\geq s$
\begin{equation}
E^{\mu_{|s,x;t,x_t}}[c(t,\cdot)] = a(t,x_t).
\end{equation}
Let us take a countable dense subset $\mathcal{Q}_{0,T}$ of $[0,T]$. There thus exists a null set $\mathcal{N}\subseteq[0,T]$ such that for every $s\in\mathcal{Q}_{0,T}$, $\mu$ a.e. $x$ and every $t\notin \mathcal{N}$ the above equation holds. We may therefore conclude \eqref{eq:aequalc} using \cite[Corollary 2.4]{revuz-yor:91} and taking $s\to t^-$. From \eqref{eq:Mderivative}, we observe that for all $s\in [0,T]$ and $\mu$ a.e. $x$, for Lebesgue a.e. $t$
\begin{equation*}
h^\mu_{s,x}(t,x_t) = E^{\mu_{|[0,s],x;t,x_t}}[e(t,\cdot)].
\end{equation*}
Equation \eqref{eq:lastidentity} thus follows from Corollary \ref{cor:last}.
\end{proof}
\subsection{Comparison of our Results to those of Fischer \textit{et al}  \cite{budhiraja-dupuis-etal:12,fischer:12}}

We have already noted in the introduction that one may infer a variational representation of the relative entropy from  \cite{budhiraja-dupuis-etal:12,fischer:12} by assuming that the coefficients of the underlying stochastic process are independent of the empirical measure in these papers. The assumptions in \cite{fischer:12} on the underlying process $P$ are both more general and more restrictive than ours. His assumptions are more general insofar as the coefficients of the SDE may depend on the past history of the process and the diffusion coefficient is allowed to be degenerate. However our assumptions are more general insofar as we only require $P$ to be the unique (in the sense of probability law) weak solution of the SDE, whereas \cite{fischer:12} requires $P$ to be the unique strong solution of the SDE. Of course when both sets of assumptions are satisfied, one may infer that the expressions for the relative entropy are identical. 
\section{Appendix}
\subsection{Proof of Theorem \ref{thm:DG}}\label{sect:DGproof}
The fact that, if $\alpha\subseteq\sigma$, then $R_\alpha(\mu||P) \leq R_\sigma(\mu||P)$, follows from Lemma \ref{lem:DonskVara}. We prove the first expression \eqref{eq:projlimit1} in the case $s=0$, $t=T$ (the proof of the second identity \eqref{eq:projlimit2} is analogous). We employ Large Deviations Theory to do this. 
\begin{definition}
A series of probability laws $\Gamma^N$ on some topological space $\Omega$ equipped with its Borelian $\sigma$-algebra is said to satisfy a strong Large Deviation Principle with rate function $I:\Omega\to\R$ if for all open sets $O$,
\begin{equation*}
\linf{N}N^{-1}\log \Gamma^N(O) \geq - \inf_{x\in O} I(x)
\end{equation*}
and for all closed sets $F$
\begin{equation*}
\lsup{N}N^{-1}\log\Gamma^N(F) \leq -\inf_{x\in F} I(x).
\end{equation*}
If furthermore the set $\lbrace x: I(x) \leq \alpha\rbrace$ is compact for all $\alpha\geq 0$, we say that $I$ is a good rate function.
\end{definition}
We define the following empirical measures.
\begin{definition}
For $x\in \T^N, y\in \T^N_\sigma$, let 
\begin{align*}
\hat{\mu}^N(x) = \frac{1}{N}\sum_{1\leq j\leq N} \delta_{x^j} \in \M(\T), \;\;\;\;
\hat{\mu}_\sigma^N(y) = \frac{1}{N}\sum_{1\leq j\leq N} \delta_{y^j} \in \M(\T_\sigma).
\end{align*}
\end{definition}
Clearly $\hat{\mu}_\sigma^N(x_\sigma)  = \pi_\sigma(\hat{\mu}^N(x))$.  The image law $P^{\otimes N}\circ (\hat{\mu}^N)^{-1}$ is denoted by $\Pi^N_{s,t} \in \M(\M(\T))$.  Similarly, for $\sigma\in\J$, the image law of $P^{\otimes N}_\sigma \circ (\hat{\mu}_\sigma^N)^{-1}$ on $\M(\T_\sigma)$ is denoted by $\Pi^N_\sigma \in  \M(\M(\T_{\sigma}))$. Since $\T$ and $\T_\sigma$ are Polish spaces, we have by Sanov's Theorem (see \cite[Theorem 6.2.10]{dembo-zeitouni:97}) that
$\Pi^N$ satisfies a strong Large Deviation Principle with good rate function $R(\cdot ||P)$. Similarly, $\Pi^N_\sigma$ satisfies a strong Large Deviation Principle on $\M(\T_\sigma)$ with good rate function $R_{\F_\sigma}(\cdot ||P)$.

We now define the projective limit $\undt{M}(\T)$. If $\alpha,\gamma\in\J$, $\alpha \subset \gamma$, then we may define the projection $\pi^\M_{\alpha\gamma}:\M(\T_\gamma) \to \M(\T_\alpha)$ as $\pi^{\M}_{\alpha\gamma}(\xi) := \xi\circ\pi_{\alpha\gamma}^{-1}$. An element of $\undt{\M}(\T)$ is then a member $\otimes_\sigma\zeta(\sigma)$ of the cartesian product $\otimes_{\sigma\in\J}\M(\T_\sigma)$ satisfying the consistency condition $\pi^{\M}_{\alpha\gamma}(\zeta(\gamma)) = \zeta(\alpha)$ for all $\alpha\subset\gamma$. The topology on $\undt{\M}(\T)$ is the minimal topology necessary for the natural projection $\undt{\M}(\T) \to \M(\T_\alpha)$ to be continuous for all $\alpha\in\J$. That is, it is generated by open sets of the form
\begin{equation}
A_{\gamma,O} = \lbrace \otimes_\sigma \zeta(\sigma) \in \undt{\M}(\T): \zeta(\gamma) \in O \rbrace,\label{eq:generatingsets}
\end{equation}
for some $\gamma\in\J$ and  open $O$ (with respect to the weak topology of $\M(\T_\gamma)$). 

We may continuously embed $\M(\T)$ into the projective limit $\undt{\M}(\T)$ of its marginals, letting $\iota$ denote this embedding. That is, for any $\sigma\in\J$, $(\iota(\mu))(\sigma) = \mu_{\sigma}$.  We note that $\iota$ is continuous because $\iota^{-1}(A_{\gamma,O})$ is open in $\M(\T)$, for all $A_{\gamma,O}$ of the form in \eqref{eq:generatingsets}. We equip $\undt{\M}(\T)$ with the Borelian $\sigma$-algebra generated by this topology. The embedding $\iota$ is measurable with respect to this $\sigma$-algebra because the topology of $\M(\T)$ has a countable base. The embedding induces the image laws $(\Pi^N\circ\iota^{-1})$ on $\M(\undt{\M}(\T))$. For $\sigma\in\J$, it may be seen that $\Pi^N_\sigma = \Pi^N\circ\iota^{-1}\circ(\pi^\M_\sigma)^{-1} \in \M(\M(\T_\sigma))$, where $\pi^{\M}_\sigma( \otimes_\alpha \mu(\alpha)) = \mu(\sigma)$.

It follows from \cite[Thm 3.3]{dawson-gartner:87} that $\Pi^N\circ\iota^{-1}$ satisfies a Large Deviation Principle with rate function 
 $\sup_{\sigma\in\J} R_\sigma(\mu||P).$ However we note that $\iota$ is $1-1$, because any two measures $\mu,\nu\in \M(\T)$ such that $\mu_\sigma = \nu_\sigma$ for all $\sigma\in\J$ must be equal. Furthermore $\iota$ is continuous.  Because of Sanov's Theorem , $(\Pi^N)$ is exponentially tight (see \cite[Defn 1.2.17, Exercise 1.2.19]{dembo-zeitouni:97} for a definition of exponential tightness and proof of this statement). These facts mean that we may apply the Inverse Contraction Principle \cite[Thm 4.2.4]{dembo-zeitouni:97}  to infer $\Pi^N$ satisfies an LDP with the rate function $\sup_{\sigma\in\J} R_{\sigma}(\mu||P)$. Since rate functions are unique \cite[Lemma 4.1.4]{dembo-zeitouni:97}, we obtain the first identity in conjunction with Sanov's Theorem. The second identity \eqref{eq:projlimit2} follows similarly. 
 
We may repeat the argument above, while restricting to $\sigma\subset\mathcal{Q}_{s,t}$. We obtain the same conclusion because the $\sigma$-algebra generated by $(\mathcal{F}_{\sigma})_{\sigma\subset\mathcal{Q}_{s,t}}$, is the same as $\F_{s,t}$. The last identity follows from the fact that, if $\alpha\subseteq\sigma$, then $R_\alpha(\mu||P) \leq R_\sigma(\mu||P)$.

\bibliographystyle{plain}
\bibliography{odyssee}

\begin{thebibliography}{10}

\bibitem{baladron-fasoli-etal:11}
Javier Baladron, Diego Fasoli, Olivier Faugeras, and Jonathan Touboul.
\newblock Mean field description of and propagation of chaos in recurrent
  multipopulation networks of hodgkin-huxley and fitzhugh-nagumo neurons.
\newblock Technical report, arXiv, 2011.
\newblock Submitted to the Journal of Mathematical Neuroscience.

\bibitem{ben-arous-guionnet:95}
G.~Ben-Arous and A.~Guionnet.
\newblock Large deviations for langevin spin glass dynamics.
\newblock {\em Probability Theory and Related Fields}, 102(4):455--509, 1995.

\bibitem{boue-dupuis:98}
Michelle Boue and Paul Dupuis.
\newblock A variational representation for certain functionals of brownian
  motion.
\newblock {\em The Annals of Probability}, 26(4):1641--1659, 1998.

\bibitem{budhiraja-dupuis-etal:12}
A.~Budhiraja, P.~Dupuis, and Fischer M.
\newblock Large deviation properties of weakly interacting processes via weak
  convergence methods.
\newblock {\em Annals of Probability}, 40(1):74--102, 2012.

\bibitem{dawson-gartner:87}
Donald Dawson and Jurgen Gartner.
\newblock Large deviations from the mckean-vlasov limit for weakly interacting
  diffusions.
\newblock {\em Stochastics}, 20, 1987.

\bibitem{dembo-zeitouni:97}
A.~Dembo and O.~Zeitouni.
\newblock {\em Large deviations techniques}.
\newblock Springer, 1997.
\newblock 2nd Edition.

\bibitem{desvillettes-villani:01}
L.~Desvillettes and C.~Villani.
\newblock On the trend to global equilibrium in spatially inhomogeneous
  entropy-dissipating systems. part 1: The linear fokker-planck equation.
\newblock {\em Communications in Pure and Applied Mathematics}, 54(1), 2001.

\bibitem{donsker-varadhan:83}
M.D. Donsker and S.R.S Varadhan.
\newblock Asymptotic evaluation of certain markov process expectations for
  large time, iv.
\newblock {\em Communications on Pure and Applied Mathematics}, XXXVI:183--212,
  1983.

\bibitem{dupuis-ellis:97}
Paul Dupuis and Richard~S. Ellis.
\newblock {\em A Weak Convergence Approach to the Theory of Large Deviations}.
\newblock {John Wiley \& Sons}, 1997.

\bibitem{faugeras-maclaurin:13b}
Olivier Faugeras and James MacLaurin.
\newblock A large deviation principle for networks of rate neurons with
  correlated synaptic weights.
\newblock Technical report, INRIA, February 2013.

\bibitem{fischer:12}
Markus Fischer.
\newblock On the form of the large deviation rate function for the empirical
  measures of weakly interacting systems.
\newblock {\em Archiv}, 2012.

\bibitem{freidlin-wentzell:98}
M.I. Freidlin and A.D. Wentzell.
\newblock {\em Random perturbations of dynamical systems}, volume 260.
\newblock Springer Verlag, 1998.

\bibitem{fritelli:00}
Marco Fritelli.
\newblock The minimal entropy martingale measure and the valuation problem in
  incomplete markets.
\newblock {\em Mathematical finance}, 10(1):39--52, 2000.

\bibitem{georgiou-lindquist:03}
Tryphon~T. Georgiou and Anders Lindquist.
\newblock Kullback-leibler approximation of spectral density functions.
\newblock In {\em Proceedings of the 42nd IEEE Conference on Decision and
  Control, Maui Hawaii USA}, 2003.

\bibitem{grandits-rheinlander:02}
Peter Grandits and Thorsten Rheinlander.
\newblock On the minimal entropy martingale measure.
\newblock {\em The Annals of Probability}, 2002.

\bibitem{josselin-garnier-yang:13}
George~Papanicolaou Josselin~Garnier and Tzu-Wei Yang.
\newblock Large deviations for a mean field model of systemic risk.
\newblock {\em SIAM Journal of Financial Mathematics}, 4(1):151--184, 2013.

\bibitem{karatzas-shreve:91}
Ioannis Karatzas and Steven~E. Shreve.
\newblock {\em Brownian motion and stochastic calculus}, volume 113 of {\em
  Graduate Texts in Mathematics}.
\newblock Springer-Verlag, New York, second edition, 1991.

\bibitem{lassalle:12}
Remi Lassalle.
\newblock Invertibility of adapted perturbations of the identity on abstract
  wiener space.
\newblock {\em Journal of Functional Analysis}, 262:2734--2776, 2012.

\bibitem{moynot-samuelides:02}
O.~Moynot and M.~Samuelides.
\newblock {Large deviations and mean-field theory for asymmetric random
  recurrent neural networks}.
\newblock {\em Probability Theory and Related Fields}, 123(1):41--75, 2002.

\bibitem{plastino-miller-etal:97}
A.R. Plastino, H.G. Miller, and A.~Plastino.
\newblock Minimum kullback entropy approach to the fokker-planck equation.
\newblock {\em Physical Review E}, 56(4):3927--3934, 1997.

\bibitem{revuz-yor:91}
Daniel Revuz and Marc Yor.
\newblock {\em Continuous Martingales and Brownian Motion}.
\newblock Springer-Verlag, 2 edition, 1991.

\bibitem{ustunel:09}
Ali~Suleyman Ustunel.
\newblock Entropy, invertibility and variational calculus of the adapted shifts
  on wiener space.
\newblock {\em Journal of Functional Analysis}, 257:3655--3689, 2009.

\bibitem{yu-mehta:09}
Sun Yu and Prashant Mehta.
\newblock The kullback-leibler rate metric for comparing dynamical systems.
\newblock In {\em Joint 48th IEEE Conference on Decision and Control and 28th
  Chinese Control Conference, Shanghai}, 2009.

\end{thebibliography}
\end{document}